\documentclass[reqno,centertags,12pt]{amsart}
\usepackage{amsmath,amsthm,amscd,amssymb,latexsym,verbatim}

\usepackage{graphicx,epsf,cite}

-\marginparwidth \oddsidemargin -4mm \evensidemargin \oddsidemargin

\newtheorem{theorem}{Theorem}[section]

\newtheorem{lemma}[theorem]{Lemma}
\newtheorem{corollary}[theorem]{Corollary}
\theoremstyle{definition}
\newtheorem{definition}[theorem]{Definition}

\theoremstyle{remark}

\newcounter{smalllist}

\DeclareMathOperator{\diam}{diam}

\DeclareMathOperator*{\dist}{dist}

\DeclareMathOperator*{\sgn}{sgn}

\allowdisplaybreaks
\numberwithin{equation}{section}

\newcommand{\lb}{\label}

\newcommand{\beq}{\begin{equation}}
\newcommand{\eeq}{\end{equation}}

\newcommand{\bal}{\begin{align}}
\newcommand{\eal}{\end{align}}
\newcommand{\bals}{\begin{align*}}
\newcommand{\eals}{\end{align*}}

\newcommand{\bbN}{{\mathbb{N}}}
\newcommand{\bbR}{{\mathbb{R}}}

\newcommand{\bbP}{{\mathbb{P}}}

\newcommand{\bbZ}{{\mathbb{Z}}}

\newcommand{\bbT}{{\mathbb{T}}}

\newcommand{\calS}{{\mathcal S}}

\newcommand{\calC}{{\mathcal C}}

\newcommand{\eps}{\varepsilon}

\newcommand{\tht}{\theta}

\newcommand{\til}{\tilde}

\begin{document}
\title[Transition fronts for bistable and ignition reactions]
{Existence and non-existence of transition fronts \\ for bistable and ignition reactions}

\author{Andrej Zlato\v s}
\address{\noindent Department of Mathematics \\ University of
Wisconsin \\ Madison, WI 53706, USA \newline Email:
zlatos@math.wisc.edu}


\begin{abstract}
We study reaction-diffusion equations in one spatial dimension and with general (space- or time-) inhomogeneous mixed bistable-ignition reactions.  For those satisfying a simple quantitative hypothesis, we prove existence and uniqueness of transition fronts, as well as convergence of ``typical'' solutions to the unique transition front (the existence part even extends to mixed bistable-ignition-monostable reactions).  These results also hold for all pure ignition reactions without any other hypotheses, but not for all pure bistable reactions.  In fact, we find examples of either spatially or temporally periodic pure bistable reactions (independent of the other space-time variable) for which we can prove non-existence of transition fronts.  These are the first such results for periodic media which are non-degenerate in a natural sense, and the spatially periodic examples also prove a conjecture from \cite{DHZ}. 
\end{abstract}

\maketitle

\section{Introduction} \lb{S1}

We study reaction-diffusion equations
\beq\lb{1.1}
u_t=u_{xx}+f(x,u)
\eeq
and 
\beq\lb{1.1a}
u_t=u_{xx}+f(t,u)
\eeq
in one spatial dimension.  
These equations are used to model a host of natural processes such as combustion, population dynamics, pulse propagation in neural networks, or solidification dynamics.
We will consider here the cases of either {\it spatially} \eqref{1.1} or {\it temporally} \eqref{1.1a} {\it inhomogeneous mixed bistable-ignition reactions}.  We are primarily interested in general (non-periodic) reactions, but our results are new even in the periodic case.

For homogeneous media, one usually considers bistable reactions to have $\til\tht\in(0,1)$ such that $f(0)=f(\til\tht)=f(1)=0$, with $f<0$ on $(0,\til\tht)$ and $f>0$ on $(\til\tht,1)$, while ignition reactions have  $f=0$ on $(0,\til\tht)$ and $f>0$ on $(\til\tht,1)$.  It is also standard to consider $f$ non-increasing near 0 and 1
(and sometimes even $f'(1)<0$, along with $f'(0)<0$ for bistable $f$).
One is then interested in solutions $0\le u\le 1$ which transition between the (stable) equilibria $u\equiv 0$ and $u\equiv 1$,  modeling  invasions of one  equilibrium of the relevant physical process by another.   Typically these include solutions evolving from initial data which are {\it spark-like} (with $\lim_{|x|\to\infty} u(0,x)=0$), or {\it front-like} (with $\lim_{x\to\infty} u(0,x)=0$ and $\liminf_{x\to-\infty} u(0,x)>\til\tht$).  It is customary to also assume $\int_0^1 f(u)du>0$, so that  solutions  which are initially above some $\beta>\til\tht$ on a large enough $\beta$-dependent interval converge locally uniformly to 1 as $t\to\infty$ (i.e., they propagate).  One is then interested in the nature of the transition from 0 to 1.  (Note that the roles of 0 and 1 are  reversed if $\int_0^1 f(u)du<0$ for bistable $f$.)  

The study of transitions between equilibria of reaction-diffusion equations has seen a lot of activity since the seminal papers of Kolmogorov, Petrovskii, Piskunov  \cite{KPP} and Fisher \cite{Fisher} (who studied homogeneous reactions). 
We are here interested in this question for $f$ which also depends on $x$ or $t$, and we will also relax the requirement for a single sign change of $f(x,\cdot)$ or $f(t,\cdot)$ in $(0,1)$.  We will therefore assume the following hypothesis.  Let us consider only \eqref{1.1} for the time being; \eqref{1.1a} will be treated afterwards.

\medskip
{\it Hypothesis (H):  $f$ is Lipschitz with constant $K\ge 1$, 
\beq\lb{1.2}
f(x,0)=f(x,1)=0 \qquad \text{for $x\in\bbR$,}  
\eeq
and there is $\tht>0$ such that for each $x\in\bbR$, $f$ is non-increasing in $u$ on $[0,\tht]$ and on $[1-\tht,1]$. 
Moreover, there are $0<\tht_1\le\tht_0<1$ and Lipschitz functions $f_0,f_1:[0,1]\to\bbR$  with $f_0\le f_1$,
\[
\begin{split}
f_0(0)=f_0(1) & =  f_1(0)  =f_1(1)=0, 
\\ f_0\le 0    \text{ on $(0,\tht_0)$} \qquad& \text{and}\qquad  f_0>0   \text{ on $(\tht_0,1)$,}
\\ f_1\le 0    \text{ on $(0,\tht_1)$} \qquad& \text{and}\qquad  f_1>0   \text{ on $(\tht_1,1)$,}
\end{split}
\]
\beq \lb{1.2c}
\int_0^1 f_0 ( u) du> 0, 
\eeq
such that
\[
f_0(u)\le f(x,u)\le f_1(u) \qquad \text{for $(x,u)\in \bbR\times [0,1]$}.
\]
}
\medskip

\begin{definition} \lb{D.1.0}
(i) We call any $f$ satisfying (H) a {\it BI reaction} (i.e., bistable-ignition).  

(ii) If $f$ is a BI reaction and $f_1< 0$ on $(0,\tht_1)$, then $f$ is a {\it bistable reaction}.  If there is also an increasing function $\gamma:[0,\infty)\to [0,\infty)$ and for each $x\in\bbR$ there is $\til \tht_x\in[\tht_1,\tht_0]$ such that 
\[
\sgn(u-\til\tht_x) f(x,u)\ge \gamma \big(\dist(u,\{0,\til\tht_x,1\}) \big)
\]
 for $u\in[0,1]$, then $f$ is a {\it pure bistable reaction}. 
 
(iii) If $f$ is a BI reaction and $f_0=0$ on $(0,\tht_0)$, then $f$ is an {\it ignition reaction}.   If there are also $\gamma$ and $\til \tht_x$ as in (ii) such that now $f(x,u)=0$ for $u\in[0,\til \tht_x]$ and
\[
f(x,u)\ge \gamma \big(\dist(u,\{\til\tht_x,1\}) \big)
\]
 for $u\in[\til\tht_x,1]$, then $f$ is a {\it pure ignition reaction}. 
\end{definition}

{\it Remark.}  We note that if instead $\tht=0=\tht_1$ in (H), then $f$ is a {\it mixed bistable-ignition-monostable reaction}, and it is a {\it pure monostable reaction} if (iii) above holds with $\til\tht_x\equiv 0$.\smallskip

Let us now briefly review some of the relevant literature for bistable and ignition reactions in one dimension (their mixtures, allowed here, may not have been studied before).  In these papers, \eqref{1.2c} need not always be assumed for bistable reactions and other hypotheses may be included. There is also a large body of work on monostable reactions in one dimension, as well as on all  reaction types in several dimensions, and the interested reader can consult \cite{Berrev, Xin2, XinBook} for reviews of these results and other related developments.


A useful tool in the study of the evolution of solutions of reaction-diffusion equations can often be special solutions called {\it transition fronts}.  These are entire solutions $w$ of \eqref{1.1}, 
defined in \cite{Matano, Shen} for some special situations and later in \cite{BH3} in more generality, satisfying
\beq\lb{1.3}
\lim_{x\to-\infty} w(t, x+x_t)=1 \qquad\text{and}\qquad \lim_{x\to\infty} w(t, x+x_t)=0
\eeq
uniformly in $t\in\bbR$, with $x_t:=\max \{x\in\bbR\,|\, w(t,x)=\tfrac 12\}$ (which includes existence of $x_t$).  This is the definition of the {\it right-moving} transition front, while the {\it left-moving} one is defined with 1 and 0 swapped.  Note that we consider here only fronts connecting 0 and 1, not those connecting other equilibria of the PDE (if they exist).  Also, we will only consider here fronts with $0\le w\le 1$, since we do not assume anything about $f(x,u)$ for $u\notin[0,1]$.  It is easy to show, however, that  if $f(x,u)\ge 0$ for $u< 0$ and  $f(x,u)\le 0$ for $u>1$, then any transition front satisfies $0< w< 1$ \cite{BH3,ZlaInhomog}.  Finally, we note that the definition we use here is that of transition fronts with a {\it single interface}.  The definition in \cite{BH3}, restricted to one dimension, allows any bounded-in-time number of interfaces, but most papers studying transition fronts in one dimension use the (very natural) single interface version.

In media where there exists a unique right-moving and a unique left-moving transition front (up to a translation in $t$), one can sometimes show that typical solutions converge to their time-shifts as $t\to\infty$.  The simplest such case are homogeneous media $f(x,u)=f(u)$, where transition fronts are known to be unique for ignition and bistable reactions, and take the form of {\it traveling fronts} $w(t,x)=W(x-ct)$ (right-moving) and $w(t,x)=W(-x-ct)$ (left-moving),  with a unique {\it front speed} $c>0$ and the {\it front profile} $W$ solving $W''+cW'+f(W)=0$ on $\bbR$ and having the limits $\lim_{s\to-\infty} W(s)=1$ and $\lim_{s\to\infty} W(s)=0$.  

The situation is slightly more complicated for spatially periodic media, where existence and uniqueness of {\it pulsating fronts} (first defined in \cite{SKT}, these are transition fronts satisfying $u(t+ \tfrac pc,x)=u(t,x-p)$ with $p$ the spatial period of $f$ and $c$ the front speed,  whose profile is time-periodic in a moving frame) 
has been proved for fairly general ignition reactions \cite{BH} but only for some special cases of bistable reactions.  This includes near-homogeneous reactions  \cite{VakVol} (see also \cite{Xin} for a related result), reactions with a constant $\til\tht_x$   (i.e., $\tht_1=\tht_0$ in (H)) \cite{NolRyz}, those for which \eqref{1.1} has no stable periodic steady states between 0 and 1  \cite{FanZha}, and those with small or large spatial periods  \cite{DHZ,DHZ2} (hence our Theorem \ref{T.1.1}(i) below is new even for periodic bistable reactions).  There is a good reason for such limitations: while uniqueness holds at least for non-stationary pulsating fronts if we also assume $f_1'(0)<0$ and $f_0'(1)<0$ \cite{DHZ2}, existence does not even for pure bistable reactions, as we show in Theorem \ref{T.1.1}(iii) below (and therefore also prove a conjecture from \cite{DHZ}).   

Another reason for the added difficulties in the inhomogeneous bistable case is the fact  that solutions may stop propagating and stationary fronts may exist, although not when \eqref{1.2c} holds.  This can naturally happen when $\int_0^1 f(x,u)du$ changes sign as $x$ varies \cite{DHZ2}, but  it can even happen for periodic pure bistable reactions with $\int_0^1 f(x,u)du>0$ for all $x\in\bbR$.  
For instance, we can take $v(x):=\tfrac 12-\tfrac 1\pi \arctan x$ and $g(u):=-v''(\tan(\tfrac \pi 2 -\pi u))$ (so that $v''+g(v)=0$ and $g$ is pure bistable with $\til\tht=\tfrac 12$ and $\int_0^1 g(u)du=0$).  Then we take any $x$-periodic $f$ with $\int_0^1 f(x,u)du>0$ for each $x\in\bbR$ such that $f(x,u)=g(u)$ for $(x,u)\in\bbR\times([0,\tfrac 16]\cup[\tfrac 14,\tfrac 34]\cup[\tfrac 56,1])$ as well as for $(x,u)\in(-\sqrt 3,-1)\times (\tfrac 34,\tfrac 56)$ and $(x,u)\in(1,\sqrt 3)\times (\tfrac 16,\tfrac 14)$.  Such pure bistable $f$ (not satisfying \eqref{1.2c}) easily exists and satisfies $v''(x)+f(x,v(x))=0$ for all $x\in\bbR$ because $v((-\sqrt 3,-1))=(\tfrac 34,\tfrac 56)$ and $v((1,\sqrt 3))=(\tfrac 16,\tfrac 14)$.
We refer the reader to \cite{DHZ2,LewKee, FifPel,Xin} and the references therein for further studies of such {\it wave-blocking} phenomena for bistable reactions.

As for non-periodic media, it was proved in \cite{MRS, NolRyz, MNRR} for ignition reactions of the form $f(x,u)=a(x)g(u)$ with some bounded $a\ge 1$ and a pure ignition $g$ (in particular, $\tht_1=\tht_0$), that exponentially decaying front-like solutions  converge to a unique right-moving front in $L^\infty(\bbR)$ as $t\to\infty$, while spark-like ones converge to it in $L^\infty({\bbR^+})$ and to a unique left-moving one in $L^\infty({\bbR^-})$.  This was extended to general ignition reactions satisfying a non-vanishing condition in \cite{ZlaGenfronts} (see Section \ref{S6} below).  For bistable reactions, these results again hold for $f(x,u)=a(x)g(u)$ with $a\ge 1$ and a pure bistable $g$ (in particular, $\tht_1=\tht_0$) 
\cite{NolRyz}, and existence of transition fronts was proved earlier for general near-homogeneous bistable reactions in \cite{VakVol}.  

On the other hand, one can easily construct situations in which transition fronts (connecting 0 and 1) do not exist, even if $f$ satisfies (H).  A simple example is a homogeneous reaction with $f(\tfrac 12)=0$ which is bistable when restricted to $u\in[0,\tfrac 12]$ (with a unique front speed $c'$) as well as when restricted to $u\in[\tfrac 12,1]$ (with a unique front speed $c''$), 
and $f(u+\tfrac 12)<f(u)$ for $u\in(0,\tfrac 12)$ (so $c''<c'$).  In that case it is easy to show that for typical solutions, the transition $0\to \tfrac 12$ propagates with speed $c'$ while the transition $\tfrac 12 \to 1$ propagates with the slower speed $c''$, creating a linearly-in-$t$ growing ``terrace'' on which $u(t,\cdot)\sim\tfrac 12$.  This and more general such situations were recently studied in \cite{DGM}.  Of course, such reactions are in some sense degenerate, being made of two or more bistable (or other type) reactions ``glued'' end-to-end.  They thus do not resolve the abovementioned question  of whether transition fronts  must always exists for general ``non-degenerate'' (i.e., pure) bistable reactions in one dimension. (For pure ignition reactions this can be answered in the affirmative using the general ignition reactions result from \cite{ZlaGenfronts} --- see Theorem \ref{T.1.1}(ii) below. For pure monostable reactions  the answer is negative \cite{NRRZ}.)
 
In the present paper we prove that existence and uniqueness of transition fronts holds for general inhomogeneous mixed bistable-ignition reactions which satisfy a simple quantitative hypothesis, and that in this case  exponentially decaying front-like and spark-like solutions again converge to these fronts as $t\to\infty$.  The same result  holds for all pure ignition reactions, without the extra hypothesis. On the other hand,  we also show that this hypothesis is not only technical.  In fact, we construct an example of a spatially  periodic pure bistable reaction for which no transition fronts exist, thus resolving the above question of their existence for pure bistable reactions in the negative.  (We note that the latter holds only in the sense of \eqref{1.3}, for fronts connecting 0 and 1.  Fronts connecting other equilibrium solutions $0\le u^-< u^+\le 1$ of \eqref{1.1} may still exist, such as a front connecting 0 and $\tfrac 12$ and another connecting $\tfrac 12$ and 1 in the example from the previous paragraph.)  This example is, to the best of our knowledge, the first of a non-degenerate (in the sense from the previous paragraph) {\it periodic} reaction of any kind in one dimension for which no transition fronts exist, since the monostable reaction examples from \cite{NRRZ} are not periodic.  (We note that non-existence of {\it traveling fronts} was previously proved for bistable reactions on some cylinders with dumbbell-shaped cross-sections and an additional shear flow \cite{BHbist}.  The reaction is homogeneous in these and the result hinges instead on the special cross-section and on the added flow.  It seems that transition fronts connecting 0 and 1 should not exist in these examples either, although those connecting other equilibria of the PDE in question again may exist.)

Here is our first main result, containing the above claims for \eqref{1.1}.  


\begin{theorem} \lb{T.1.1}
Let $f$ be a BI reaction from (H), with $c_0$ the unique front speed for $f_0$. 

(i) Assume that $f_1(u)< \tfrac {c_0^2}4 u$ for all $u\in(0,\tht_1']$, with $\tht_1'\in[\tht_0,1)$ given by  $\int_{\tht_1}^{\tht_1'} f_0(u) du=0$.  Then there exists a unique (up to translation in $t$) right-moving transition front $w$ for \eqref{1.1} (and a unique left-moving one $\til w$), which then also satisfies $w_t>0$.  Moreover,  solutions with exponentially decaying initial data converge to time shifts of $w, \til w$ (see Definition \ref{D.1.0a} below).

(ii)  The claims in (i) hold  (without the hypothesis on $f_1$) if $f$ is a pure ignition reaction.

(iii)  There exists an $x$-periodic pure bistable reaction $f$ such that there is no (right- or left-moving) transition front for \eqref{1.1} in the sense of  \eqref{1.3}.  In fact, there are $\eps_0,m>0$ such that any solution $0\le u\le 1$ with $u(0,\cdot)\ge \tfrac 12\chi_{[-m,m]}$ and $\lim_{x\to\infty} u(0,x)=0$ takes values within $[\eps_0,1-\eps_0]$ on spatial intervals whose lengths grow linearly in time as $t\to\infty$.
\end{theorem}

{\it Remarks.}  1. The proof of (i) shows that its existence part extends to  mixed bistable-ignition-monostable reactions with $f_1$ satisfying the hypothesis from (i).
\smallskip

2. Note that the hypothesis of (i) is automatically satisfied when $\tht_1=\tht_0$ (as in \cite{NolRyz, MRS,MNRR}), or when $\tht_1$ is close enough to $\tht_0$ (e.g., when $\int_{\tht_1}^{4K(4K-c_0^2)^{-1}\tht_1}f_0(u) du>0$).
\smallskip

3. In (i), the limits in \eqref{1.3} are uniform in $f$ for any fixed $f_0,f_1,K$, while those in  \eqref{1.5} and \eqref{1.6} below are uniform in $f,u$ for any fixed $f_0,f_1,K,\tht,Y,\mu,\beta$ (with $l_{f_0,\beta}$ in Definition \ref{D.1.0a}(b) only depending on $f_0,\beta$).  In (ii) this is true if we also fix $\gamma$ from Definition \ref{D.1.0}. 
\smallskip

4. The first claim in (iii) thus proves the conjecture from \cite{DHZ} about existence of such  reactions.  The second claim  describes the reason behind non-existence of transition fronts for these reactions.
\smallskip



\begin{definition} \lb{D.1.0a}
Let $w,\til w$ be some right- and left- moving transition fronts for \eqref{1.1}. We say that {\it solutions with exponentially decaying initial data converge to time shifts of $w,\til w$} if the following hold for any $Y,\mu>0$, $\beta>\tht_0$, and $a\in\bbR$.
 
(a) If $u$ is a (front-like) solution of  \eqref{1.1} with  
\[
\beta \chi_{(-\infty,a]}(x)\le u(0,x) \le e^{-\mu(x-a-Y)},
\]
then there is $\tau_u$ such that 
\beq \lb{1.5}
\lim_{t\to\infty} \|u(t,\cdot) - w(t+\tau_u,\cdot)\|_{L^\infty} =0
\eeq
(and similarly for $\til w$ and $u$ exponentially decaying as $x\to -\infty$).

(b) There is $l_{f_0,\beta}<\infty$ such that if $L\ge l_{f_0,\beta}$ and $u$ is a (spark-like) solution of \eqref{1.1} with
\[
\beta\chi_{[a-L,a+L]}(x)\le u(0,x) \le \min\{e^{-\mu(x-a-L-Y)}, e^{\mu(x-a+L+Y)}\},
\]
then there are $\tau_u,\til \tau_u$ such that 
\beq \lb{1.6}
\lim_{t\to\infty} \|u(t,\cdot) - w(t+\tau_{u},\cdot) - \til w(t+\til \tau_{u},\cdot) + 1\|_{L^\infty} =0.
\eeq
\end{definition}

%
%

Let us now turn to the time-inhomogeneous reactions case \eqref{1.1a}.  Here we replace in (H) and in Definition \ref{D.1.0}  each $x$ by $t$, while Definition \ref{D.1.0a} refers to convergence to space shifts of $w,\til w$ and has $w(t+\tau_{u},\cdot)$ and $\til w(t+\til \tau_{u},\cdot)$ replaced by $w(t,\cdot+x_u)$ and $\til w(t,\cdot+\til x_u)$.  The definition of transition fronts is unchanged.

The time-periodic bistable reaction case was first studied in \cite{ABC} (the abstract framework of  \cite{FanZha} also applies to this case), where it was proved that a unique pulsating front (now satisfying $u(t+ p,x)=u(t,x-pc)$ with $p$ the temporal period of $f$ and $c$ the front speed) exists provided the ODE $v'=f(t,v)$ has a unique periodic solution $v:\bbR\to(0,1)$, which is also unstable.

This was extended to almost-periodic and general stationary ergodic bistable reactions in \cite{Shen2,Shen3, Shen1}, provided that there is again a single solution $v:\bbR\to(0,1)$ of the ODE $v'=g(t,v)$ (which must also be unstable) for each $g$ in the $L^\infty_{\rm loc}$-closure of the family of all time-translates of $f$.
Finally,  some general results about transition fronts in stationary ergodic media were proved in \cite{Shen}, which were then applied to show existence of a transition front  for $f(t,u)=u(1-u)(u-a(t))$, with $a(t)\in[\tfrac 38,\tfrac 58]$ a stationary ergodic process.

The study of time-inhomogeneous ignition reactions is only very recent, with \cite{ShenShen, ShenShen2} proving existence, uniqueness, and stability of transition fronts for ignition reactions with a constant $\til\tht_t$  (so $\tht_1=\tht_0$), also satisfying some additional technical hypotheses.

We now state our second main result, the time-inhomogeneous version of Theorem \ref{T.1.1}, whose part (ii) also extends \cite{ShenShen, ShenShen2} to general pure ignition reactions.

\begin{theorem} \lb{T.1.1a}
Let $f$ be a BI reaction from (H) with each $x$ replaced by $t$, with $c_0$ the unique front speed for $f_0$. 

(i) Assume that $f_1(u)< \tfrac {c_0^2}4 u$ for all $u\in(0,\tht_0]$.  Then the claims in Theorem \ref{T.1.1}(i) hold for \eqref{1.1a}, 
with uniqueness of the front up to translation in $x$ and with $w_x<0$ instead of $w_t>0$.

(ii)  The claim in (i) holds  (without the hypothesis on $f_1$) if $f$ is a pure ignition reaction.

(iii)  Theorem \ref{T.1.1}(iii) holds for \eqref{1.1a}, with $f$ being a $t$-periodic pure bistable reaction.
\end{theorem}

{\it Remark.} The remarks after Theorem \ref{T.1.1} are also valid here.
\smallskip

We close this introduction with an application of our results to the cases of periodic and stationary ergodic reactions.

\begin{corollary}\lb{C.1.5}
The following hold under the hypotheses of one of Theorem \ref{T.1.1}(i), Theorem~\ref{T.1.1}(ii), Theorem \ref{T.1.1a}(i), Theorem \ref{T.1.1a}(ii).

(i)  If $f$ is spatially/temporally periodic, then the unique transition front is a pulsating front.

(ii)  If $f$ is stationary ergodic with respect to spatial/temporal translations (see Section~\ref{S8} for the precise definition of this), then the unique transition front almost surely has a deterministic asymptotic speed $c>0$ in the sense of $\lim_{|t|\to\infty} \tfrac{x_t}t=c$.
\end{corollary}

The author thanks  Peter Pol\' a\v cik for a helpful discussion about Theorem \ref{T.1.1}(iii).  He also acknowledges partial support by NSF grant DMS-1056327.

\section{Proof of Theorem \ref{T.1.1}(i)} \lb{S2}

This follows the lines of a similar proof for ignition reactions in \cite{ZlaGenfronts}.  That proof was done for the PDE 
\beq\lb{2.1}
u_t  = \nabla \cdot (A(x) \nabla u) +  q(x)\cdot \nabla u + f(x,u)
\eeq
with $x\in \bbR\times\bbT^{N-1}$, a uniformly elliptic periodic $n\times n$ matrix $A$ and a divergence-free periodic vector field $q$, but not necessarily periodic $f$.  We will not consider this setting here.

The  existence part of the proof will be done in detail, since it has non-trivial differences from \cite[Section 2]{ZlaGenfronts}.  Once this is obtained, proofs of uniqueness of the transition front and of convergence of typical solutions to its time-shifts are virtually identical to those in \cite[Sections 3 and 4]{ZlaGenfronts}.  (The ignition property is used in them several times, but it is immediately obvious that $f(x,\cdot)$ being non-increasing on $[0,\tht]$ for each $x\in\bbR$ suffices instead.)   We will therefore only sketch these two parts of the proof here, both for the convenience of the reader as well as for later reference in the proof of Theorem \ref{T.1.1a}(i).  


\pagebreak

\noindent
{\bf Existence of a front}
\vskip 3mm

Pick any ($f_0$-dependent) $\eps_0\in(0,\tht_0)$ such that $\int_0^{1-\eps_0} f_0(u)du>0$ and $1-\eps_0$ is greater than any point of maximum of $\tfrac{f_0(u)}u$.
Using \eqref{1.2c}, it is easy to construct $v:\bbR\to[0,1]$ satisfying $v''+f_0(v)\ge 0$, supported on $\bbR^-$, and equal to $1-\eps_0$ for $x\ll -1$.
One can take $v\equiv 1-\eps_0$ on $(-\infty, 0]$, let $v''+f_0(v)=0$ (with $v(0)=1-\eps_0$ and $v'(0)=0$) on $(0,r)$, where $r>0$ is smallest such that $v(r)=0$, and let $v\equiv 0$ on $[r,\infty)$ (then we shift $v$ by $r$ to the left).  The existence of $r$ follows from multiplying $v''+f_0(v)=0$ by $v'$ and integrating over $(0,x)$, which yields $\tfrac 12 v'(x)^2= F_0(1-\eps_0)-F_0(v(x))$, with 
\[
F_0(u):=\int_0^u f_0(s)ds.
\]
  Since $F_0(u)<F_0(1-\eps_0)$ for $u\in[0, 1-\eps_0)$ due to $F_0(1-\eps_0)=\int_0^{1-\eps_0} f_0(u)du>0$, we see that $v'$ cannot change sign before $v$ hits 0. So $v'$ stays negative, and then $v$ must hit 0 at some finite $r$ because $F_0(u)<F_0(1-\eps_0)$ for $u\in[0, 1-\eps_0)$.

We now let $u_n$ be the solution of \eqref{1.1} with initial condition $u_n(0,x)=v(x+n)$.  Then $f\ge f_0$ and $1-\eps_0>\tht_0$, together with well known spreading results \cite{AW2,FM}, imply that $\lim_{t\to\infty} u_n(t,x)=1$ locally uniformly.
Hence there is (minimal) $\tau_n$ such that $u_n(\tau_n,0)=\tfrac 12$, and then finite speed of propagation (e.g., $u_n(t,x)\le e^{-\sqrt\xi(x+n-2\sqrt\xi t)}$ for $\xi:=\max_{u\in(0,1]} \tfrac{f_1(u)}u$, since the exponential is a super-solution of \eqref{1.1} when we define $f(x,u)=0$ for $u>1$)  easily shows $\tau_n\to\infty$.  We let $\til u_n(t,x):=u_n(t+\tau_n,x)$, so that $\til u_n$ solves \eqref{1.1} on $(-\tau_n,\infty)\times\bbR$, with $\til u_n(0,0)=\tfrac 12$.  Parabolic regularity now shows that some subsequence of $\til u_n$ converges in $C^{1,2}_{\rm loc}$ to an entire solution $w$ of \eqref{1.1} with $w(0,0)=\tfrac 12$.   We also have $w_t\ge 0$ due to $(u_n)_t\ge 0$, which follows from $(u_n)_t(0,\cdot)\ge 0$ and the maximum principle for $(u_n)_t$.  To show that   $w$ is indeed a transition front, we now only need to prove that the limits \eqref{1.3} hold uniformly in $t\in\bbR$.
(This and $w(0,0)=\tfrac 12$ also imply $w_t\not\equiv 0$, and the strong maximum principle for $w_t$ then proves $w_t>0$ as well.)
 This will in turn be proved by showing that
\beq\lb{2.4}
 \sup_{n\in\bbN\,\&\,t\ge T_\eps} \diam \{x\in\bbR\,|\, u_n(t,x)\in[\eps,1-\eps]\} <\infty
\eeq
 for each $\eps>0$ and some $n$-independent $T_\eps<\infty$.

We now pick $\zeta<\tfrac{c_0^2}4$ and $\tht_1''>\tht_1'$ (both depending only on $f_0,f_1$) so that  
\beq\lb{2.5}
f_1(u)<  \zeta u \qquad \text{ for $u\in(0,\tht_1'']$},
\eeq
and let $c_\zeta:=2\sqrt\zeta$ and  $c_\xi:= (\xi+\zeta)\zeta^{-1/2}$.  It is well known that $c_0\le c_1\le 2\sqrt\xi$ (with $c_1$ the unique front speed for $f_1$), hence we have $\zeta< \xi$ and $c_\zeta<c_0\le  c_\xi$.  

Finally, for each $n\in\bbN$ and $t\ge 0$ we let
\[
X_n(t):= \max \{x \in\bbR \,|\, u_n(t,x)\ge\tht_1''\},
\]
\[
Y_n(t):= \min \{y \in\bbR \,|\, u_n(t,x)\le e^{-\sqrt\zeta(x-y)} \text{ for all $x\in \bbR$}\}.
\]
We note that the proof in \cite{ZlaGenfronts} (see also Section \ref{S6}) defined $X_n(t)$ to be the largest $x$ for which  $f(x,u)< \zeta u$ does not hold for all $u\in(0,u_n(t,x))$ (which is then smaller than our $X_n(t)$), but our definition will suffice here.  Also note that $X_n$ and $Y_n$ are both non-decreasing because $(u_n)_t\ge 0$, and we have $X_n(0)=X_0(0)-n$ and $Y_n(0)=Y_0(0)-n$.  Since $\tht_1''$ is smaller than any point of maximum of $\tfrac{f_0(u)}u$ (due to $\zeta<\max_{u\in(0,1]} \tfrac{f_0(u)}u$, which follows from $\zeta<\tfrac{c_0^2}4$) we obtain $\tht_1''<1-\eps_0$.  Hence $X_n(t)$ is finite, while $Y_n(t)$ is finite by  the following crucial lemma.

\begin{lemma} \lb{L.2.1}
(i) For any $n$ and $t\ge t'\ge 0$ we have
\beq\lb{2.6}
Y_n(t)-Y_n(t')\le c_\xi(t-t').
\eeq
If also $X_n(t)\le Y_n(t')$, then in fact 
\beq\lb{2.7}
Y_n(t)-Y_n(t')\le c_\zeta(t-t').
\eeq

(ii) For every $\eps>0$ there is $r_\eps<\infty$ such that for any $n$ and $t \ge t' \ge 0$ we have
\beq\lb{2.8}
\inf_{|x-X_n(t')|\le c_0(t-t')-r_\eps}  u_n(t,x)  \ge 1-\eps.
\eeq
This $r_\eps$ only depends on $\eps,f_0,f_1,K$.
\end{lemma}

\begin{proof}
(i)  The first claim follows from $e^{-\sqrt\zeta(x-Y_n(t')-c_\xi(t-t'))}$ being a super-solution of \eqref{1.1}.  The second claim follows from $w(t,x):=e^{-\sqrt\zeta(x-Y_n(t')-c_\zeta(t-t'))}$ satisfying $w_t=w_{xx}+\zeta w$, while $u_n$ is a sub-solution of this PDE on $(t',t)\times(X_n(t),\infty)$ due to \eqref{2.5}, the definition of $X_n$, and due to $X_n$ being non-decreasing (note that $w\ge 1 >u_n$ on $(t',t)\times(-\infty,X_n(t)]$ because $X_n(t)\le Y_n(t')$).  

(ii)  Note that \eqref{2.8} will follow from $f(x,u)\ge f_0(u)$ and well-known spreading results (i.e., spreading with speed $c_0$ for $u_t=u_{xx}+f_0(u)$ \cite{AW2,FM}) once we show for 
each $L<\infty$ existence of $T<\infty$ (depending on $L,\tht_1'',f_0,f_1,K$) such that under the hypotheses of (ii) we have
\beq\lb{2.9}
\inf_{|x-X_n(t')|\le L}  u_n(t'+T,x)  \ge \tht_1''.
\eeq
(Here $\tht_1''$ can be replaced by any constant larger than $\tht_1$.  Moreover, an $L$ such that \eqref{2.9} indeed implies \eqref{2.8} only depends on $\tht_1'',f_0$, while $\tht_1''$  only depends on $f_0,f_1$.  Hence $r_\eps$ will only depend on $\eps,f_0,f_1,K$.)  We will now prove \eqref{2.9} for any fixed $L$.

First we claim that $u_n(t',x)\ge \tht_1$ for $x\le  X_n(t')$.  This is because $f(x,\tht_1)\le 0$, so the set $I_n(t):=\{x\,|\, u_n(t,x)\ge \tht_1\}$ cannot acquire new connected components due to the maximum principle, and because $(u_n)_t\ge 0$,  so $I_n(t)$ cannot split into several connected components either.  Since $I_n(0)$ is some interval $(-\infty,\iota-n]$, it follows that $I_n(t)$ is some interval $(-\infty,\iota_n(t)]$.

Assume now that \eqref{2.9} does not hold for some $L$.  Then for each $k\in\bbN$ we can find $n_k$ and  $(t'_k,x_k)\in [0,\infty)\times[-L,L]$ such that $u_{n_k}(t'_k+k, X_{n_k}(t'_k)+x_k)<\tht_1''$.  Then each $ w_k(t,x):= u_{n_k}(t+t'_k,x+X_{n_k}(t'_k))$ satisfies \eqref{1.1} on $\bbR^+\times\bbR$, with $f$ replaced by $g_k(x,u):=f(x+X_{n_k}(t'_k),u)$.  Parabolic regularity, $f_0\le f\le f_1$, and $f$ being $K$-Lipschitz show that a subsequence of $w_k$ converges in $C^{1,2}_{\rm loc}(\bbR^+\times\bbR)$ to some solution $\til w$ of \eqref{1.1}, with $f$ replaced by some $K$-Lipschitz $g$ such that $f_0\le g\le f_1$.  Moreover, $\til w_t>0$, $\til w(0,\cdot)\ge \tht_1 \chi_{\bbR^-}$, $\til w(0,0)\ge \tht_1''$, and $w(x):=\lim_{t\to\infty} \til w(t,x)$ solves $w''+g(x,w)=0$ on $\bbR$ and  satisfies $w(x_0)\le \tht_1''$ for some $|x_0|\le L$.  Since $w\le 1$, it also follows that $w<1$.

We thus obtain $w''+f_0(w)\le 0$ and $ \tht_1 \chi_{\bbR^-}\le w<1$ as well as $w(0)\in[ \tht_1'',1)$.  Multiplying the former by $w'$ and integrating over $(a,0)$, with $a\in[-\infty,0)$ smallest such that $w'$ does not change sign on $(a,0)$ (hence $w'(a)=0$) yields
\[
\sgn(w(0)-w(a)) \left [\frac {w'(0)^2}2+F_0(w(0))- F_0(w(a)) \right] \le 0.
\]
From $w(0)\ge \tht_1''$, $w(a)\ge \tht_1$, and $\int_{\tht_1}^{\tht_1''} f_0(u)du>0$ we obtain 
 \[
 \sgn(F_0(w(0))-F_0(w(a)))=\sgn(w(0)-w(a)),
 \]
so we must have $w(a)\ge w(0)$.  

If $a>-\infty$, we let $a'\in[-\infty,a)$ be smallest such that $w'$ does not change sign on $(a',a)$, and the same argument yields
\[
\sgn(w(a)-w(a'))  [F_0(w(a))- F_0(w(a'))] \le 0.
\]
Since $w(a)\in[\tht_1'',1)$,  $w(a')\in[\tht_1,1]$,  and $\int_{\tht_1}^{\tht_1''} f_0(u)du>0$, we see that this is only possible if $w(a')=w(a)$.  But then $a'=-\infty$ and $w\equiv w(a)$ on $(-\infty,a)$, a contradiction with $w''+f_0(w)\le 0$ because $f_0>0$ on $[\tht_1'',1)$.

If now $a=-\infty$, we must have $f_0(w(-\infty))\le 0$ (and $w(-\infty)\ge \tht_1''$) which leaves us with $w(-\infty)=1$.  Running the above argument on $(-\infty,b)$, with $b\in[0,\infty]$ largest such that $w'$ does not change sign on $(-\infty,b)$, then yields
\[
\sgn(w(b)-1)  [F_0(w(b))- F_0(1)] \le 0.
\]
The properties of $f_0$ now force $w(b)=1$.  Hence $b=\infty$ and $w\equiv 1$, a contradiction with $w(x_0)\le \tht_1''$.

This proves \eqref{2.9}.
Notice that the $T$ we obtained  is independent of $f$ because the contradiction argument can be run uniformly in all $f$ from (H) (we pick $\{(f_k,n_k,t_k',x_k)\}_{k=1}^\infty$ instead of  $\{(n_k,t_k',x_k)\}_{k=1}^\infty$).  Thus $T=T(L,\tht_1'',f_0,f_1,K)$
and as mentioned above, it follows that $r_\eps$ depends only on $\eps,f_0,f_1,K$.  
\end{proof}

Having proven the lemma, we now easily obtain
\beq\lb{2.10}
 \sup_{n\in\bbN\,\&\,t\ge 0} |Y_n(t)-X_n(t)|\le C
\eeq
for some $C=C(f_0,f_1,K)$.  The uniform bound $ X_n(t)-Y_n(t)\le C(f_0,f_1)$ is obvious from the definition of $X_n,Y_n$ (since $\eps_0,\tht_1'',\zeta$ only depend on $f_0,f_1$), so we are left with proving $ Y_n(t)-X_n(t)\le C(f_0,f_1,K)$.  

Note that the claims of Lemma \ref{L.2.1}(i) together prove
\beq\lb{2.11}
Y_n(t)-Y_n(t') \le c_\zeta(t-t') \qquad\text{when $Y_n(t)-X_n(t)\ge c_\xi(t-t')$},
\eeq
and Lemma \ref{L.2.1}(ii) shows
\beq\lb{2.12}
X_n(t)-X_n(t') \ge c_0(t-t')-r_{\eps_0}
\eeq
(recall that we have $1-\eps_0\ge \tht_1''$).  Let $S:=|Y_n(0)-X_n(0)|$ (which is independent of $n,f$) and $C:=S+c_\xi r_{\eps_0}(c_0-c_\zeta)^{-1} $.  If $t\ge 0$ is the first time such that $Y_n(t)-X_n(t)=C$ (note that $Y_n(t)-X_n(t)$ is lower semi-continuous because so is $X_n$, and $Y_n$ is continuous), then $t\ge r_{\eps_0}(c_0-c_\zeta)^{-1}$ by Lemma \ref{L.2.1}(i) and we let $t':=t-r_{\eps_0}(c_0-c_\zeta)^{-1}$.  But now \eqref{2.11} and \eqref{2.12} yield $X_n(t)-X_n(t')\ge Y_n(t)-Y_n(t')$, a contradiction with the choice of $t$.  This proves \eqref{2.10}.

Finally, let us define
\begin{align*}
Z_{n,\eps}^-(t) & := \max \{ y\in\bbR \, |\, u_n(t,x)> 1-\eps \text{ for all $x< y$} \}, \\
Z_{n,\eps}^+(t) & := \min \{ y\in\bbR \, |\, u_n(t,x)< \eps \text{ for all $x> y$} \}.
\end{align*}
Continuity of $Y_n$ and \eqref{2.10} show that the non-decreasing function $X_n$ can have jumps no longer than $2C$.  This, $(u_n)_t\ge 0$, and Lemma \ref{L.2.1}(ii) (together with $X_n(0)\le -n$ and $u_n(0,x)\ge 1-\eps_0>\tht_1''$ for $x\le -n-r$; see the construction of the initial data $v$ above), imply that there is $T_\eps$ such that $Z_{n,\eps}^-(t+T_\eps)\ge X_n(t)$ for any $t\ge 0$, and $T_\eps$ depends only on $\eps,f_0,f_1,K$.  From the definition of $Y_n$ and Lemma \ref{L.2.1}(i) we also have
\[
Z_{n,\eps}^+(t+T_\eps) \le Y_n(t+T_\eps)+ \zeta^{-1/2}|\log\eps| \le Y_n(t)+c_\xi T_\eps+ \zeta^{-1/2}|\log\eps|,
\]
so \eqref{2.10} allows us to conclude for each $t\ge 0$,
\beq\lb{2.13}
Z_{n,\eps}^+(t+T_\eps) - Z_{n,\eps}^-(t+T_\eps)\le c_\xi T_\eps + \zeta^{-1/2}|\log\eps| + C \qquad (=:L_\eps).
\eeq
But this is precisely \eqref{2.4}, and the proof is finished.  

Notice that this also shows that the upper bound $L_\eps$ on the left-hand side of \eqref{2.4} only depends on $\eps,f_0,f_1,K$, so as claimed in Remark 3 after Theorem \ref{T.1.1}, the limits in \eqref{1.3} are indeed uniform in all $f$ satisfying (H) with some fixed $f_0,f_1,K$.

Notice also that so far we used neither $\tht>0$ nor $\tht_1>0$.  Hence existence of fronts extends to mixed bistable-ignition-monostable reactions.

\vskip 4mm
\noindent
{\bf Uniqueness of the front and convergence of typical solutions to it}
\vskip 3mm

As mentioned above, these proofs are essentially identical to their analogs in \cite[Sections~3 and 4]{ZlaGenfronts}.  We only sketch them and refer the reader to \cite{ZlaGenfronts} for any  details skipped here.

Replace $\eps_0$ from the existence proof by the minimum of itself and $\tfrac\tht 2$ (hence it now depends on $f_0,\tht$).
Then let $v$ and $u:=u_0$ be from the existence proof (i.e., $u$ solves \eqref{1.1} with $u(0,\cdot)=v$), and let
\begin{align*}
X_u(t):= & \max \{x \in\bbR \,|\, u(t,x)\ge\tht_1''\},
\\ Y_u(t):= & \min \{y \in\bbR \,|\, u(t,x)\le e^{-\sqrt\zeta(x-y)} \text{ for all $x\in \bbR$}\},
\\ Z_{u,\eps}^-(t)  := & \max \{ y\in\bbR \, |\, u(t,x)> 1-\eps \text{ for all $x< y$} \},
\\ Z_{u,\eps}^+(t) := & \min \{ y\in\bbR \, |\, u(t,x)< \eps \text{ for all $x> y$} \}
\end{align*}
for $t\ge 0$.  We also define
\[
Z_u(t) :=  Z_{u,\eps_0}^-(t),
\]
and note that \eqref{2.10}, Lemma \ref{2.1}, and $Z_{n,\eps}^-(t+T_\eps)\ge X_n(t)$ proved above show
\beq\lb{2.21}
\sup_{t\ge T_{\eps_0}} |Y_u(t)-Z_u(t)|\le C_2,
\eeq
with $C_2=C_2(f_0,f_1,K,\tht)$ and $T_\eps=T_\eps(f_0,f_1,K)$.

Let now $0\le w\le 1$ be any transition front for \eqref{1.1}, define $X_w(t),Y_w(t),Z_{w,\eps}^-(t),Z_{w,\eps}^+(t),Z_w(t)$ as above but with $w$ in place of $u$ and for any $t\in\bbR$ (here $Y_w(t)$ might, in principle, be $\infty$).  Also define 
\[
L_w:=\sup_{t\in\bbR} \left\{ Z_{w,\eps_0}^+(t) -Z_{w,\eps_0}^-(t) \right\},
\]
which is finite because $w$ is a transition front.

First, \cite[Lemma 3.1]{ZlaGenfronts} shows 
\beq\lb{2.22}
\sup_{t\in \bbR} |Y_w(t)-Z_w(t)|\le \til C_2,
\eeq
with $\til C_2$ depending on $f_0,f_1,K,\tht$ (and also on $L_w$ if $w_t\not>0$).  Consider first the case $w_t>0$.  Then \eqref{2.22} is obtained by letting for $h>0$,
\[
Y_{w,h}(t):= \min \{y \in\bbR \,|\, w(t,x)\le h+e^{-\sqrt\zeta(x-y)} \text{ for all $x\in \bbR$}\} <\infty,
\]
and proving for all small $h>0$,
\beq\lb{2.23}
 \sup_{t\in\bbR} |Y_{w,h}(t)-X_w(t)|\le C_2(f_0,f_1,K)
\eeq
(then we take $h\to 0$, and afterwards conclude \eqref{2.22} as we did \eqref{2.21}).  Finally, \eqref{2.23} is obtained as in the existence proof, using Lemma \ref{L.2.1} for $X_w,Y_{w,h}$ and any $t\ge t'$, which holds for any $h>0$ such that
\beq\lb{2.24}
f_1(u)<  \zeta (u-h) \qquad \text{ for $u\in(h,\tht_1'']$}.
\eeq
This is true for all small enough $h>0$ due to \eqref{2.5}.

If now $w_t\not>0$ and $h>0$ is small enough, then Lemma \ref{L.2.1}(i) holds for $X_w,Y_{w,h}$  and $t\ge t'$, but with $X_w(t)\le Y_{w,h}(t')$ replaced by $\sup_{t'\le s\le t} X_w(s)\le Y_{w,h}(t')$.  Also, Lemma \ref{L.2.1}(ii) easily holds with \eqref{2.8} replaced by
\beq\lb{2.25}
\inf_{x\le X_w(t')-L_w+ c_0(t-t')-r_\eps}  w(t,x)  \ge 1-\eps.
\eeq
This is because $\eps_0<\tht_0\le \tht_1''<1-\eps_0$, so  \eqref{2.9} can be replaced in the proof by the obvious
\[
\inf_{x\le X_w(t')- L_w}  w(t',x)  \ge \tht_1''.
\]
This version of Lemma \ref{L.2.1} yields \eqref{2.23}, with $C_2$ also depending on $L_w$.

Next, \cite[Lemma 3.2]{ZlaGenfronts} shows that for each $\eps>0$ there is $\delta>0$, depending also on $f_0,f_1,K,\tht$ (and also on $L_w$ if $w_t\not>0$), such that the following holds for any $t_0\ge 1$, $t_1\in\bbR$, and $t\ge t_0$:
\beq\lb{2.26}
\text{if $\pm[w(t_1,\cdot)-u(t_0,\cdot)]\le\delta$, then $\pm[w(t+t_1-t_0,\cdot)-u(t,\cdot)]\le\eps$.}
\eeq
Of course, $u,w$ are the particular solutions of \eqref{1.1} considered here (in particular, one needs to use that they decay exponentially  as $x\to\infty$). The proof of the $+$ case (without loss assume $t_1=t_0$, otherwise shift $w$ in time) is via the construction of a super-solution of \eqref{1.1} of the form
\beq\lb{2.27}
z_+(t,x):= u \left(t+ \frac\eps\Omega \left(1-e^{-\sqrt\zeta(c_0-c_\zeta)(t-t_0)/4} \right) ,x \right) + b_\eps e^{-\sqrt\zeta(x-Y_w(t_0)-c_\zeta(t-t_0))/2},
\eeq
with $\Omega$ large so that $|u_t|\le\Omega$ for $t\ge 1$ (such $\Omega=\Omega(K)$ exists by parabolic regularity) and $b_\eps>0$ small and depending also on $f_0,f_1,K,\tht$ (and also on $L_w$ if $w_t\not>0$).  That such $b_\eps$ exists follows from $u_t>0$, the strong maximum principle for $u_t$, and (recall that $\eps_0\le\tfrac\tht 2$) $\sup_{t\ge 1} \{Z_{u,\tht/2}^+(t)-Z_{u,\tht/2}^-(t)\}<\infty$ --- which together show that $u_t(t,x)$ is uniformly positive where $u(t,x)\in[\tfrac\tht 2,1-\tfrac\tht 2]$ --- as well as from $f$ being non-increasing in $u$ on $[0,\tht]$ and on $[1-\tht,1]$ (we also let $f(x,u)\le 0$ for $u> 1$).  Note that a crucial property of $z_+$ is that the second term travels with speed $c_\zeta$, which is smaller than the lower bound $c_0$ on the speed of propagation of the first term.  Hence $z_+(t,\cdot+x_t)-u(t+\tfrac\eps\Omega,\cdot+x_t)$, with $x_t:=\max \{x\in\bbR\,|\, u(t,x)=\tfrac 12\}$, converges locally uniformly to 0 as $t\to\infty$.  A simple argument \cite{ZlaGenfronts} then concludes the $+$ case of \eqref{2.26} with $\delta$ depending on $b_\eps$ (specifically, $\delta=b_{\eps}^2$ for that particular choice of $b_\eps$, and then one obtains \eqref{2.26} with $2\eps$ instead of $\eps$).

The proof of the $-$ case of \eqref{2.26} is similar, using the sub-solution
\beq\lb{2.28}
z_-(t,x):= u \left(t- \frac\eps\Omega \left(1-e^{-\sqrt\zeta(c_0-c_\zeta)(t-t_0)/4} \right) ,x \right) - b_\eps e^{-\sqrt\zeta(x-Y_u(t_0)-c_\zeta(t-t_0))/2}
\eeq
as well as Lemma \ref{L.2.1}(ii) for $w$ (the latter is needed because $\lim_{x\to-\infty} z_-(t,x)=-\infty$).

These estimates now easily show (see \cite[Lemma 3.3]{ZlaGenfronts}) that 
\beq\lb{2.29}
\tau_w:=\inf\{\tau \in\bbR \,|\, \liminf_{t\to\infty} \inf_{x\in\bbR}[w(t+\tau,x)-u(t,x)]\ge 0\}
\eeq
is a finite number, and hence also
\beq\lb{2.30}
\liminf_{t\to\infty} \inf_{x\in\bbR}[w(t+\tau_w,x)-u(t,x)]\ge 0.
\eeq
Then it is shown in  \cite[Lemma 3.4]{ZlaGenfronts} that in fact
\beq\lb{2.31}
\lim_{t\to\infty} \|w(t+\tau_w,\cdot)-u(t,\cdot)\|_{L^\infty}= 0.
\eeq
Indeed, if this were false, then \eqref{2.26}, \eqref{2.30}, \eqref{2.21}, \eqref{2.22}, and the strong maximum principle would imply
\[
\liminf_{t\to\infty} \inf_{|x-x_t|\le L} [w(t+\tau_w,x)-u(t,x)]>0
\]
for any $L<\infty$.  This, \eqref{2.26}, the definition of $\tau_w$, and $f$ being non-increasing in $u$ on $[0,\tht]$ and on $[1-\tht,1]$ can be shown to yield a contradiction (also using parabolic regularity).

Hence each transition front must converge in $L^\infty$ to some time-shift of $u$ as $t\to\infty$.  Finally, this convergence is shown in \cite[Lemma 3.5]{ZlaGenfronts} to be uniform in all $f$ satisfying (H) (for any fixed $f_0,f_1,K,\tht$) and all $w$ with $L_w\le C$ (for any fixed $C<\infty$).  Indeed, if this were not true, one could obtain a counter-example to \eqref{2.31} by passing to a subsequence of more and more slowly converging couples $u,w$ as above (each with its own $f$; this again uses parabolic regularity and $f$ being $K$-Lipschitz).

Since this uniformity includes any translations of $f$ in $x$, we obtain that if $w_1,w_2$ are two  transition fronts for $f$, the solutions $u_n$ of \eqref{1.1} with initial conditions $u_n(0,x):=v(x+n)$ converge uniformly quickly (in $n$) as $t\to\infty$ to some time translates (by $\tau_{1,n}$ and $\tau_{2,n}$) of $w_1,w_2$.  Obviously $\tau_{1,n},\tau_{2,n}\to-\infty$ as $n\to\infty$, which together with the stability result \eqref{2.26} shows that for any $t'\in\bbR$ and $\eps>0$, there is $\tau_{t',\eps}$ such that 
\[
\sup_{t\ge t'\,\&\,x\in\bbR}|w_1(t,x)-w_2(t+\tau_{t',\eps},x)|<\eps.
\]
Since $t'\in\bbR$ and $\eps>0$ are arbitrary, it follows that $w_1(\cdot,\cdot)\equiv w_2(\cdot-\tau,\cdot)$ for some $\tau\in\bbR$.  Thus there is a unique transition front (up to translation in $t$), which then must be the one constructed in the existence proof.  That front satisfies $w_t>0$ and has $L_w$ uniformly bounded in $f$ (for any fixed $f_0,f_1,K,\tht$), hence we find that, in fact, the constants in the above results do not depend on $L_w$.

This proves the uniqueness claim of Theorem \ref{T.1.1}(i).  The proof of the convergence claim for front-like solutions  is very similar to the uniqueness proof, but with $u$ now being the unique transition front, while $w$ being the front-like solution (so the notation from Definition \ref{D.1.0a} is reversed).  There are only two significant differences.  The first is that $Y_w$ must now be defined with  $\sqrt\zeta$ replaced by $\mu$ so that it is finite, and $\sqrt\zeta$ is replaced by $2\mu$ in \eqref{2.27}
(this uses $\mu\le\tfrac 12 \sqrt\zeta$, which can be assumed without loss).  The second is that while now we do not have the uniform limits \eqref{1.3}, we can instead easily bound $w$ from above by  the new super-solution \eqref{2.27} and from below by the original sub-solution \eqref{2.28} (with two different $t_0$).  This, \eqref{2.21} (with $t\in\bbR$), and \eqref{2.6} (which now holds with $c_\xi:=(K+\mu^2)\mu^{-1}$) then prove \eqref{2.22} (with $t\ge 0$), even though the crucial estimate \eqref{2.7} does not anymore hold for the new $Y_w$ and some $c_\zeta<c_0$.

The proof for spark-like solutions  is identical, but restricted to $x\in\bbR^+$ (and then to $x\in\bbR^-$ and the unique left-moving front).  Finally, the second claim in Remark 3 after Theorem \ref{T.1.1} is also proved as in \cite{ZlaGenfronts} --- if it were false, one could use parabolic regularity to construct a reaction satisfying the hypotheses but not the result on convergence of front-like or spark-like solutions to the transition fronts.

\section{Proof of Theorem \ref{T.1.1}(ii)} \lb{S6}

This is an immediate corollary of \cite[Theorem 1.3]{ZlaGenfronts}.   The latter is the same result for \eqref{2.1} and ignition reactions (see Definition \ref{D.1.0})  satisfying the following hypothesis (which we state here in the case of \eqref{1.1}, with $\tht_0$ from (H) and $c_0$ the unique front speed for $f_0$):

\smallskip
{\it There are $\zeta<\tfrac{c_0^2}4$ and $\eta>0$ such that 
\beq\lb{6.1}
\inf_{x\in\bbR \,\&\, u\in[\alpha_f(x),\tht_0]} \sup_{|y-x|\le\eta^{-1}} f(y,u)\ge\eta,
\qquad\text{where }
\alpha_f(x):= \inf\{u\in(0,1) \,|\, f(x,u)\ge \zeta u\}.
\eeq
}

This hypothesis  automatically holds for all pure ignition reactions (even without the $\sup$ and with $f(y,u)$ replaced by $f(x,u)$), with any $\zeta\in(0,\tfrac{c_0^2}4)$ and $\eta$ depending on $\zeta,\tht_1,K,\gamma$ (the latter from Definition \ref{D.1.0}).  We also note that in the proof of \cite[Theorem 1.3]{ZlaGenfronts}, $X_n$ is replaced by the smaller
\[
\tilde X_n(t):= \max \{x \in\bbR \,|\, u_n(t,x)\ge\alpha_f(x)\},
\]
that the extra hypothesis on $f_1$ from Theorem \ref{T.1.1}(i) can be replaced by \eqref{6.1}  thanks to the fact that any bounded solution to $0=u_{xx}+f(x,u)$ with $f\ge 0$  must be constant, 
and 
that the existence part of that result extends to mixed ignition-monostable reactions.

\section{Proof of Theorem \ref{T.1.1}(iii)} \lb{S3}


We start with a periodic stationary solution $p$ of \eqref{1.1} with $f=f_0$, where $f_0$ is any homogeneous pure bistable reaction.  It is well known that such solutions are obtained by solving the ODE $p''+f_0(p)=0$ on $\bbR$, with any $p(0)\in[\tht_0,\tht_0')$ and $p'(0)=0$, where $\tht_0'\in(\tht_0,1)$ is given by $\int_{0}^{\tht_0'} f_0(u)du=0$.  It is easy to show (by multiplying the ODE by $p'$ and integrating on any interval where $p'$ does not change sign) that $p(\bbR)=[P,p(0)]$, where $\int_P^{p(0)} f_0(u)du=0$.
We pick $p(0)=\tfrac 14(\tht_0+3\tht_0')$ and denote $M$ the period of the corresponding solution $p$.  Next we let $m>0$ be such that $p\ge \tfrac 14(2\tht_0+2\tht_0')$ on $[-m,m]$.  We let $\kappa$ be a Lipschitz constant for $f_0$, and for any $\delta>0$ let $a\in(0,\tht_0)$ be such that if $w_t=w_{xx}$ and $w(0,\cdot)\ge \tht_0 \chi_{(-m,m)}$, then $w(\delta,\cdot)\ge a e^{\kappa\delta} \chi_{(-m-M,m+M)}$.  This means that whenever $u$ solves \eqref{1.1} with $f\ge f_0$ and $u(t',\cdot)\ge \tht_0 \chi_{(A-m,A+m)}$ for some $A\in\bbR$, then 
\beq\lb{3.1}
u(t'+\delta,\cdot)\ge a  \chi_{(A-m-M,A+m+M)}.
\eeq

Next, for any given $K<\infty$ we pick any (Lipschitz) even-in-$x$ pure bistable $f\ge f_0$ such that $f(x,u)=f_0(u)$ when $u\notin (\tfrac a 2,p(x))$ and 
\[
f(x,u)=f_0(u)+K \dist\left(u,\left\{\frac a 2,\frac{3\tht_0+\tht_0'}4\right\}\right)
\]
when $|x-nM|\le m$ for some $n\in\bbZ$ and $u\in[\tfrac a 2,\tfrac14(3\tht_0+\tht_0')]$.  If $K$ is large enough, this can be done so that $f$ is indeed pure bistable.
It is now clear from \eqref{3.1} that if $u$ solves \eqref{1.1} and $u(t',\cdot)\ge \tht_0 \chi_{(nM-m,nM+m)}$ for some $n\in\bbZ$, then we have
\[
u(t'+2\delta,\cdot)\ge \tht_0  \chi_{((n-1)M-m,(n+1)M+m)}
\]
provided $K$ is large enough.  This immediately yields for such $K$ and $j=1,2,\dots$,
\beq\lb{3.2}
u(t'+2j\delta,\cdot)\ge \tht_0  \chi_{((n-j)M-m,(n+j)M+m)}.
\eeq

We now pick any $ \delta>0$ such that $4\delta\sqrt\kappa<M$, then $K$ as above (so that \eqref{3.2} holds) and fix the corresponding $f$.  If $u$ is any transition front for \eqref{1.1} (we only need to consider right-moving ones because $f$ is even in $x$), we have $u(0,\cdot)\ge \tht_0 \chi_{(nM-m,nM+m)}$ for some $n\in\bbZ$.  From \eqref{3.2} we then get for $j=1,2,\dots$,
\beq\lb{3.3}
u(2j\delta,\cdot)\ge \tht_0 \chi_{((n-j)M-m,(n+j)M+m)}.
\eeq

On the other hand, we have
\beq\lb{3.4a}
u(t,x)\le w(t,x):=p(x) + e^{-\sqrt{\kappa}(x-A-2\sqrt\kappa t)}
\eeq
for some large $A<\infty$ and all $(t,x)\in\bbR^+\times\bbR$.  This is true for $t=0$ because $u(0,\cdot)$ is bounded and $\lim_{x\to\infty} u(0,x)=0<P$, and then it holds for $t>0$ because  $w$ is a super-solution of \eqref{1.1} (recall that $f(x,u)=f_0(u)$ for $u\ge p(x)$ and $\kappa$ is a Lipschitz constant for $f_0$):
\[
w_t-w_{xx} -f(x,w) =f_0(p(x))-f_0(w)+ \kappa e^{-\sqrt{\kappa}(x-A-2\sqrt\kappa t)} \ge 0.
\]
This means that we have (in fact, for any solution of \eqref{1.1} with $\limsup_{x\to\infty} u(0,x)<P$, and some large enough $A$)
\beq\lb{3.4}
u \left( t, 2\sqrt\kappa \,t + A + \frac 1\kappa  \log \frac 2{1-p(0)} \right) \le \frac{1+p(0)}2.
\eeq

This and \eqref{3.3} now show for $\eps_0:= \min \{\tht_0,\tfrac{1-p(0)}2  \}$ and  $j=1,2,\dots$ that $u(2j\delta,\cdot)$ takes values within $[\eps_0,1-\eps_0]$ on some interval of length
\[
(M-4\delta\sqrt\kappa)j +nM -  A - \frac 1\kappa  \log \frac 2 {1-p(0)}.
\]
Since  $4\delta\sqrt\kappa<M$, it follows that
\eqref{1.1} with this pure bistable $f$ does not have any transition fronts (connecting 0 and 1).

The second claim
is proved identically.  Indeed, in the above argument we only needed  that $u(t',\cdot)\ge \tht_0\chi_{(-m,m)}$ for some $t'\in\bbR$ (we can pick $\tht_0=\tfrac 12$),  and $\limsup_{x\to\infty} u(t',x)<P$.

\section{Proof of Theorem \ref{T.1.1a}(i)} \lb{S4}

This proof is similar to the one of Theorem \ref{T.1.1}(i), with space-shifts replacing time-shifts at various points.  Its existence part is slightly different, while the other two parts are essentially identical.

\vskip 4mm
\noindent
{\bf Existence of a front}
\vskip 3mm

We again let $u_n$ solve \eqref{1.1a}, but this time with initial condition $u_n(-n,x)=v(x)$ (where $\eps_0,v$ are from the existence part of the proof of Theorem \ref{T.1.1}(i)).  We then let $\xi_n$ be maximal such that $u_n(0,\xi_n)=\tfrac 12$ (from $f\ge f_0$ we have $\lim_{n\to\infty}\xi_n=\infty$) and define $\til u_n(t,x):=u_n(t,x+\xi_n)$.  We again recover our candidate for a front $w$ (with $w(0,0)=\tfrac 12$) as a limit of a subsequence of these $\til u_n$, and it remains to prove \eqref{2.4} with $t\ge -n+T_\eps$ instead of $t\ge T_\eps$.  Note that now $(u_n)_x\le 0$, so this time $w_x<0$ will also follow.

We now pick $\zeta<\tfrac{c_0^2}4$ and $\tht_1''>\tht_0$ so that \eqref{2.5} holds, and again let $c_\zeta:=2\sqrt\zeta$ and  $c_\xi:= (\xi+\zeta)\zeta^{-1/2}$ (recall that $\xi:=\max_{u\in(0,1]} \tfrac{f_1(u)}u$).  We then take for $t\ge -n$,
\beq\lb{4.1}
X_n(t):= \max \{x \in\bbR \,|\, u_n(s,x)\ge\tht_1'' \text{ for some $s\in[-n,t]$}\},
\eeq
\beq\lb{4.2}
Y_n(t):= \min \{y \in\bbR \,|\, u_n(s,x)\le e^{-\sqrt\zeta(x-y)} \text{ for all $(s,x)\in [-n,t]\times\bbR$}\}.
\eeq
The crucial lemma is now the following.

\begin{lemma} \lb{L.4.1}
(i) Lemma \ref{L.2.1}(i) holds for any $n$ and $t\ge t'\ge -n$.

(ii) For every $\eps>0$ there is $r_\eps<\infty$ such that for any $n$ and $t \ge t' \ge -n$ we have
\beq\lb{4.3}
\inf_{x\le X_n(t')+ c_0(t-t')-r_\eps}  u_n(t,x)  \ge 1-\eps.
\eeq
This $r_\eps$ only depends on $\eps,f_0,f_1$.
\end{lemma}

\begin{proof}
(i)  This is identical to the proof of Lemma \ref{L.2.1}(i).

(ii)  This is immediate from the spreading results in \cite{AW2,FM}, $\tht_1''>\tht_0$, and $(u_n)_x\le 0$ (here $r_\eps$ depends on $\eps, f_0,\tht_1''$, and the latter depends only on $f_0,f_1$.)
\end{proof}

The rest of the existence proof carries over from Theorem \ref{T.1.1}(i), with $t\ge -n$ instead of $t\ge 0$, the constant $C$ in \eqref{2.10} only depending on $f_0,f_1$, and $Z_{n,\eps}^-(t+T_\eps)\ge X_n(t)$ (with $T_\eps:=r_\eps c_0^{-1}$) following directly from \eqref{4.3}.  In particular,  $T_\eps$ now only depends on $\eps,f_0,f_1$, hence so does the upper bound on the left-hand side of \eqref{2.4}.  This means that for any fixed $f_0,f_1$, the limits in \eqref{1.3} are uniform in all $K$ and all $f$ satisfying (H) with $x$ replaced by $t$.

We note that again we used neither $\tht>0$ nor $\tht_1>0$ so far, hence existence of fronts extends to mixed bistable-ignition-monostable reactions.

\vskip 4mm
\noindent
{\bf Uniqueness of the front and convergence of typical solutions to it}
\vskip 3mm

This is virtually identical to the same proof in Theorem \ref{T.1.1}(i), but with time-shifts of solutions replaced by space-shifts.  The only changes are the following.  The definitions of $X_u,Y_u,X_w,Y_w,Y_{w,h}$ are adjusted as in \eqref{4.1} and \eqref{4.2}, while those of $Z_u, Z_{u,\eps}^-,Z_{u,\eps}^+,Z_w,Z_{w,\eps}^-,Z_{w,\eps}^+$ stay unchanged.  Each ``$w_t>0$'' is replaced by ``$w_x<0$''.  Claims \eqref{2.26} are replaced by
\beq\lb{4.4}
\text{if $\pm[w(t_0,\cdot-x_0)-u(t_0,\cdot)]\le\delta$, then $\pm[w(t,\cdot-x_0)-u(t,\cdot)]\le\eps$.}
\eeq
In their proofs we can assume $x_0=0$ and use
\beq\lb{4.5}
z_\pm(t,x):= u \left(t,x\mp \frac\eps\Omega \left(1-e^{-\sqrt\zeta(c_0-c_\zeta)(t-t_0)/4} \right) \right) \pm b_\eps e^{-\sqrt\zeta(x-Y_w(t_0)-c_\zeta(t-t_0))/2}
\eeq
(with $Y_u(t_0)$ instead of $Y_w(t_0)$ in the $-$ case), where $\Omega$ is such that  $|u_x|\le \Omega$ for $t\ge 1$.  The time-shift $\tau_w$ is replaced by the space-shift
\[
\xi_w:=\inf\{\xi \in\bbR \,|\, \liminf_{t\to\infty} \inf_{x\in\bbR}[w(t,x-\xi)-u(t,x)]\ge 0\},
\]
and $w(t+\tau_w,x)$ is replaced by $w(t,x-\xi_w)$ in the corresponding argument.  In the last paragraph of the uniqueness proof we use initial conditions $u_n(-n,x)=v(x)$ and obtain for all $t'\in\bbR$, $\eps>0$, and some $\xi_{t',\eps}$,
\[
\sup_{t\ge t'\,\&\,x\in\bbR}|w_1(t,x)-w_2(t,x-\xi_{t',\eps})|<\eps.
\]
 This concludes the proof of uniqueness of the front (up to translation in $x$), and the claim of convergence of typical solutions to its space-shifts uses the same adjustments.

\section{Proof of Theorem \ref{T.1.1a}(ii)} \lb{S7}

This is an immediate corollary of the following result, which is an analog of \cite[Theorem 1.3]{ZlaGenfronts} for time-dependent ignition reactions.   

\begin{theorem} \lb{T.7.1}
Let $f$ be an ignition reaction, satisfying (H) with each $x$ replaced by $t$, with $c_0$ the unique front speed for $f_0$. 
Assume that there are $\zeta<\tfrac{c_0^2}4$ and $\eta>0$ such that 
\beq \lb{7.0}
\inf_{t\in\bbR \,\&\, u\in[\alpha_f(t),\tht_0]} f(t,u)\ge\eta, 
\qquad \text{where }
\alpha_f(t):= \inf\{u\in(0,1) \,|\, f(t,u)\ge \zeta u\}.
\eeq
Then the claims in Theorem \ref{T.1.1}(i) hold for \eqref{1.1a}, 
with uniqueness of the front up to translations in $x$ and with $w_x<0$ instead of $w_t>0$.
\end{theorem}


As in Section \ref{S6}, the above hypothesis  automatically holds for all pure ignition reactions, with any $\zeta\in(0,\tfrac{c_0^2}4)$ and $\eta$ depending on $\zeta,\tht_1,K,\gamma$ (the latter from Definition \ref{D.1.0}).
This proves Theorem \ref{T.1.1a}(ii), so it remains to prove Theorem \ref{T.7.1}.

In fact, we only need to prove the existence part of the result. 
This is because the remaining claims are then proved identically to Theorem \ref{T.1.1a}(i).  Moreover, the beginning and the end of the proof of the existence part are also identical to that of Theorem \ref{T.1.1a}(i).  There is, however, a difference in Lemma \ref{L.4.1}(ii) because there need not be any $\tht_1''>\tht_0$ such that \eqref{2.5} holds.  This will also require a slightly more refined part (i). 

We use the notation from the beginning of the existence part of Section \ref{S4} (recall, in particular, that $\eps_0$ only depends on $f_0$), but with \eqref{4.1} replaced by
\beq\lb{7.1}
X_n(t):= \max \{x \in\bbR \,|\, u_n(t,x)\ge\alpha_f(t)\}.
\eeq
We will also need
\beq\lb{7.2}
Z_{n}(t)  := \max \{ y\in\bbR \, |\, u_n(t,x)\ge 1-\eps_0 \text{ for all $x< y$} \}.
\eeq
Notice that we have $Z_n(t)\le X_n(t)$ due to $\alpha_f(t)<1-\eps_0$ (see the argument just before the statement of Lemma \ref{L.2.1}).  Here is the relevant version of Lemma \ref{L.2.1}, which also includes the analog of \eqref{2.10}.  

\begin{lemma} \lb{L.7.2}
(i) For any $n$ and $t\ge t'\ge -n$ we have (with $|A|$ the Lebesque measure of $A$)
\beq\lb{7.3}
Y_n(t)-Y_n(t')\le c_\zeta(t-t') + (c_\xi-c_\zeta)|\{s\in[t',t] \,|\, X_n(s)> Y_n(t')\}|.
\eeq

(ii) For every $\eps>0$ there is $r_\eps<\infty$ such that for any $n$ and $t \ge t' \ge -n$ we have
\beq\lb{7.5}
\inf_{x\le Z_n(t')+ c_0(t-t')-r_\eps}  u_n(t,x)  \ge 1-\eps.
\eeq
There is also $C$ such that
\beq\lb{7.4}
 \sup_{n\in\bbN\,\&\,t\ge -n} |Y_n(t)-Z_n(t)|\le C.
\eeq
The  $r_\eps$ only depends on $\eps,f_0$, while $C$ only depends on $f_0,f_1,K,\zeta,\eta$.
\end{lemma}

Lemma \ref{L.7.2}(ii) immediately yields $Z_{n,\eps}^-(t+T_\eps)\ge Z_n(t)$ (with $T_\eps:=r_\eps c_0^{-1}$), and as at the end of the proof of Theorem \ref{T.1.1}(i), we obtain an upper bound on the left-hand side of \eqref{2.4}.  This depends on $\eps,f_0,f_1,K,\zeta,\eta$, so the limits in \eqref{1.3} depend on $f_0,f_1,K,\zeta,\eta$.  Therefore, to finish the proof of Theorem \ref{T.7.1}, it  remains to prove the lemma.

\begin{proof}
(i)  Let $A:=\{s\ge t'\,|\, X_n(s)> Y_n(t')\}$ and let $a(t):=c_\zeta(t-t')+ (c_\xi-c_\zeta)|A\cap[t',t]|$ for any $t\ge t'$.  Then $w(t,x):=e^{-\sqrt\zeta(x-Y_n(t')-a(t))}$ satisfies $w_t=w_{xx}+(\zeta+(\xi-\zeta)\chi_A(t)) w$, while $u_n$ is a sub-solution of this PDE on $(t',\infty)\times(Y_n(t'),\infty)$ due to the definition of $X_n$ and $\alpha_f$.  Since also $w\ge 1 >u_n$ on $(t',\infty)\times(-\infty,Y_n(t')]$, we have $w\ge u_n$ on $[t',\infty)\times\bbR$, and the result follows.  

(ii)  The first claim follows as in Lemma \ref{L.4.1}(ii), so we only need to prove \eqref{7.4} (and only the inequality $Y_n(t)-Z_n(t)\le C$ because the opposite one is obvious, with $C=\zeta^{-1}|\ln(1-\eps_0)|$).  

Let $\beta>0$ be the smallest positive number such that $f_1(\beta)=\zeta\beta$, so that $\alpha_f(t)\ge \beta$ for all $f$ from (H).  Let also $\eta'>0$ be such that any $K$-Lipschitz function greater than $\eta$ on $[0,\tht_0]$ and greater than $f_0$ on $[\tht_0,1]$ is greater than $\eta'$ on $[0,1-\tfrac{\eps_0}2]$.  And let $\delta:=\tfrac 12(c_0-c_\zeta)(c_\xi-c_\zeta)^{-1}>0$.

We first claim that for any large enough $T<\infty$ there is $L_T<\infty$ (depending also on $\eps_0,\eta',\delta$) such that the following holds.  If $A\subseteq(0,T)$ satisfies $|A|\ge\delta T$ and $v_t=v_{xx}+h(t)$ on $(0,T)\times(0,L_T)$ with initial condition $v(0,\cdot)\equiv\beta$, boundary conditions $v_x(\cdot,0)\equiv v(\cdot,L_T)\equiv 0$, and  $h(t)\ge 0$ such that $h(t)\ge \eta'$ for each $t\in A$ for which $v(t,0)\le 1-\tfrac{\eps_0}2$, then $v(T',0)\ge 1-\eps_0$ for each $T'\in[\sup A,T]$.  Indeed, this follows for each $T\ge (1-\tfrac{\eps_0}2-\beta)(\eta'\delta)^{-1}$ from parabolic regularity and the fact that if $L_T$ is replaced by $\infty$, then $v$ is only a function of $t$ and we obviously have $v(T',\cdot)\ge 1-\tfrac{\eps_0}2$ for each $T'\in[\sup A,T]$.

This, $(u_n)_x\le 0$, and the comparison principle now yield for each large enough $T$ that if $t'\ge -n$ and the set $A:=\{t\in[t',t'+T] \,|\, X_n(t)> Y_n(t')\}$ satisfies $|A|\ge \delta T$, then
\beq\lb{7.6}
Z_n(t'+T)\ge X_n(\inf A)-L_T\ge Y_n(t')-L_T.
\eeq
Let us now take $T\ge 2r_{\eps_0}(c_0-c_\zeta)^{-1}$, so that $c_\zeta T+ (c_\xi-c_\zeta)\delta T\le c_0 T-r_{\eps_0}$ (then $T$ and $L_T$ depend only on $f_0,f_1,K,\zeta,\eta$).  We  find using \eqref{7.3} and \eqref{7.5} that if $|A|\le\delta T$, then
\beq\lb{7.7}
Z_n(t'+T)-Z_n(t')\ge Y_n(t'+T)-Y_n(t').
\eeq
From \eqref{7.6}, \eqref{7.7}, and $Y_n(t'+T)\le Y_n(t')+c_\xi T$ (which is due to (i)) we obtain for $j=0,1,...$
\beq\lb{7.8}
Y_n(-n+jT)- Z_n(-n+jT)\le \max\{r,L_T+c_\xi T\},
\eeq
where $r>0$ is such that $v\equiv 1-\eps_0$ on $(-\infty,-r]$ (so that $Y_n(-n)- Z_n(-n)\le r$ for each $n$).  We also have $Z_n(t)\ge Z_n(t')-r$ for $t\ge t'$ because $(u_n)_x\le 0$ and $v$ is a sub-solution of \eqref{1.1a}.  This, \eqref{7.3}, and \eqref{7.8} now yield
\[
Y_n(t)-Z_n(t) \le L_T+2c_\xi T+2r
\]
for all $t\ge -n$, finishing the proof.
\end{proof}

Note that again we used neither $\tht>0$ nor $\tht_1>0$ in the proof of the existence part of Theorem \ref{T.7.1}, so that result extends to mixed ignition-monostable reactions.

\section{Proof of Theorem \ref{T.1.1a}(iii)} \lb{S5}

We will use the following lemma, in which we let 
\[
\begin{split}
g_0(u):= &
\begin{cases}
0 & u\in[0,\frac 12], 
\\ (u-\frac 12)(1-u)(u-\frac 23) & u\in(\frac 12,1],
\end{cases}
\\
g_1(u):= &
\begin{cases}
0 & u\in[0,\frac 1{11}] \cup [\frac 12,\frac 23],
\\ K \dist(u,\{\frac 1{11},\frac 12\}) & u\in(\frac 1{11},\frac 12)
\\ (u-\frac 12)(1-u)(u-\frac 23) & u\in(\frac 23,1],
\end{cases}
\end{split}
\]
with some $K\ge 0$ (which will need to be large in (iv) below).

\begin{lemma} \lb{L.5.1}
There are $M>0$ and $a\in (0,\tfrac 1{16})$ such that the following hold.

(i) If $u_t=u_{xx}+g_0(u)$ on $(0,1)\times\bbR$ and $u(0,\cdot)\le \chi_{(-\infty,0]}+\tfrac 58 \chi_{(0,\infty)}$, then 
\beq\lb{5.1}
u(1,\cdot)\le \chi_{(-\infty,M]}+(\tfrac 58 -2a) \chi_{(M,\infty)}.
\eeq

(ii) If $u_t=u_{xx}+g_1(u)$ on $(1,4)\times\bbR$ and \eqref{5.1} holds, then for all $K\ge 0$,
\[
u(4,\cdot)\le \chi_{(-\infty,2M]}+(\tfrac 58 -a) \chi_{(2M,\infty)}.
\]

(iii) If $u_t=u_{xx}+g_0(u)$ on $(-1,2)\times\bbR$ and $u(-1,\cdot)\ge \tfrac 4{11}\chi_{(-M,M)}$, then 
\[
\min\{u(0,\cdot),u(2,\cdot)\}\ge \tfrac 3{11}\chi_{(-1,1)}.
\]

(iv) If $u_t=u_{xx}+g_1(u)$ on $(2,3)\times\bbR$ and $u(2,\cdot)\ge \tfrac 2{11}\chi_{(-1,1)}$, then for all large enough $K$,
\[
u(3,\cdot)\ge \tfrac 5{11}\chi_{(-4M,4M)}.
\]
\end{lemma}

\begin{proof}
(i)  This is obvious for any small enough $a>0$ and any large enough $M$ from $g_0(\tfrac 58)<0$.  We fix this $a$, while $M$ may still be increased to satisfy (ii,iii).

(ii)  This is obvious for the above $a$ and any large enough $M$ from $g_1(\tfrac 58-2a)=0$.

(iii)  This is obvious for any large enough $M$ from $g_0=0$ on $[0, \tfrac 12]$.

(iv)  Fixing $M$ from (i,ii,iii), this follows for any large enough $K$ from $\tfrac 1{11}<\tfrac 2{11} < \tfrac 5{11}<\tfrac 1{2}$.
\end{proof}

To prove Theorem \ref{T.1.1a}(iii), we fix $a,M,K$ from the lemma and pick $\delta\in(0,\tfrac a4)$ (then 
$e^{3\delta}< 1+a$) and any pure bistable $(x,t)$-independent  $f_0\le f_1$ as in (H)  such that $f_0\le g_0$ and
\beq\lb{5.2}
| f_j(u)-g_j(u)|\le \delta u
\eeq
for $u\in[0,1]$ and $j=0,1$ (note that such $f_0,f_1$ exist because $g_0\le g_1$ and  $\int_0^1 g_0(u)du>0$).
We also let $f(t,u)$ be a pure bistable reaction satisfying (H) with these $f_0,f_1$ and all $x$ replaced by $t$, which is time-periodic with period 4 and also satisfies
\beq\lb{5.3}
f(t,\cdot)=f_0 \text{ for $t\in [0,1]$} \qquad\text{and}\qquad f(t,\cdot)=f_1 \text{ for $t\in [2,3]$}.
\eeq

If now $u$ solves \eqref{1.1a} with $u(t',\cdot)\le \chi_{(-\infty,x']}+\tfrac 58 \chi_{(x',\infty)}$ for some $t'\in 4\bbZ$ and $x'\in\bbR$, then Lemma \ref{L.5.1}(i) and $f_0\le g_0$ show
\[
u(t'+1,\cdot)\le \chi_{(-\infty,x'+M]}+(\tfrac 58 -2a) \chi_{(x'+M,\infty)} .
\]
Then Lemma \ref{L.5.1}(ii), \eqref{5.2}, and $e^{3\delta}< 1+a$ show
\[
u(t'+4,\cdot)\le \min \left\{1,e^{3\delta} \left[ \chi_{(-\infty,x'+2M]}+(\tfrac 58 -a) \chi_{(x'+2M,\infty)} \right] \right\} \le \chi_{(-\infty,x'+2M]}+\tfrac 58 \chi_{(x'+2M,\infty)}.
\]
Iterating this, we obtain for $j=1,2,\dots$,
\[
u(t'+4j,\cdot)\le  \chi_{(-\infty,x'+2jM]}+\tfrac 58 \chi_{(x'+2jM,\infty)}.
\]

Since $u(0,\cdot)\le \chi_{(-\infty,A]}+\tfrac 58 \chi_{(A,\infty)}$ for some $A\in\bbR$ whenever $\limsup_{x\to\infty}u(0,x)<\tfrac 58$, we obtain for any transition front $u$ for \eqref{1.1a} (we only need to consider the right-moving ones because $f$ is $x$-independent) and some $A\in\bbR$,
\beq\lb{5.4}
u(4j,\cdot)\le  \chi_{(-\infty,A+2jM]}+\tfrac 58 \chi_{(A+2jM,\infty)}
\eeq
for $j=1,2,\dots$, an estimate analogous to \eqref{3.4a}.

Similarly, we can use Lemma \ref{L.5.1}(iii,iv), \eqref{5.2}, and $e^{-3\delta}> 1-a>\tfrac {4}{5}$ to show that if $u$ solves \eqref{1.1a} with $u(t'-1,\cdot)\ge \tfrac 4{11} \chi_{(x'-M,x'+M)}$ for some $t'\in 4\bbZ$ and $x'\in\bbR$, then
\[
u(t'+3,\cdot)\ge \tfrac 4{11}\chi_{(x'-4M,x'+4M)}.
\]
Iteration then again yields for $j=1,2,\dots$,
\[
u(t'-1+4j,\cdot)\ge \tfrac 4{11}\chi_{(x'-(3j+1)M,x'+(3j+1)M)}, 
\]
and one more application of Lemma \ref{L.5.1}(iii), \eqref{5.2}, and $e^{-3\delta}>\tfrac {2}{3}$ yields for $j=1,2,\dots$,
\[
u(t'+4j,\cdot)\ge \tfrac 2{11}\chi_{(x'-3jM,x'+3jM)}, 
\]
Hence for any transition front $u$ and some $B\in\bbR$ we obtain
\beq\lb{5.5}
u(4j,\cdot)\ge \tfrac 2{11}\chi_{(B-3jM,B+3jM)}
\eeq
for $j=1,2,\dots$, an estimate analogous to \eqref{3.3}.

This and \eqref{5.4} now show for $\eps_0:= \tfrac 2{11}$ and  $j=1,2,\dots$ that $u(4j,\cdot)$ takes values within $[\eps_0,1-\eps_0]$ on some interval of length  $Mj+B-A$.
Since  $M>0$, it follows that
\eqref{1.1a} with this pure bistable $f$ does not have any transition fronts (connecting 0 and 1).  

Similarly to Section \ref{S3}, the second claim is proved identically.  

\section{Proof of Corollary \ref{C.1.5}} \lb{S8}

(i)  This is immediate from uniqueness of the front and the fact that its single space/time period translate is also a transition front.  We note that if $f_1'(0)<0$ and $f_0'(1)<0$, then the result also follows from our existence of transition fronts, $c_0>0$, and \cite[Theorem 1.6]{DHZ}.

(ii)  The stationary ergodic assumption on $f$ means that there is a probability space $(\Omega,\mathcal{F},\bbP)$,  $f:\Omega\to L^\infty_{\rm loc}(\bbR\times[0,1])$ is measurable and satisfies the required hypotheses uniformly in $\omega\in\Omega$, and there is a group $\{\pi_k\}_{k\in \bbZ}$ of measure preserving transformations acting ergodically on $\Omega$ such that either $f({\pi_k\omega};x,u)=f(\omega;x-kp,u)$ or $f({\pi_k\omega};t,u)=f(\omega;t-kp,u)$ for some $p>0$.

The proof of this part is similar to \cite[Corollary 1.7]{ZlaGenfronts}.
Let us start with the space-inhomogeneous reaction case.  Let $v$ be the function from Section \ref{S2}, and let $u_m$ solve \eqref{1.1} with initial condition $u_m(0,x):=v(x-mp)$  (so that $(u_m)_t>0$). For integers $n\ge m$ define
\[
\tau_{m,n}(\omega) := \inf \big\{ t\ge 0 \,\big|\, u_m(t,x)\ge v(x-np) \text{ for all $x\in \bbR$} \big\}.
\]
As in \cite{ZlaGenfronts}, the subadditive ergodic theorem \cite{Kingman, Liggett} applies to $\tau_{m,n}$ and yields finite positive deterministic limits
\[
\tau = \lim_{n\to\infty} \frac{\tau_{0,n}(\omega)} n = \lim_{n\to \infty} \frac{\tau_{-n,0}(\omega)} n 
\]
for almost all $\omega\in\Omega$. 
Uniform convergence (in $m$ and $\omega$) of the solution $u_m$ to the front $w_{\omega}$ in $L^\infty$  (see Remark 3 after Theorem \ref{T.1.1}) then shows  that $c:=\tfrac p\tau$ is the asymptotic speed of $w_{\omega}$ as $|t|\to\infty$ for almost all $\omega\in\Omega$.

In the time-inhomogeneous reaction case we instead let $u_m$ solve \eqref{1.1a} with initial condition $u_m(mp,x):=v(x)$, and for integers $n\ge m$ define
\[
\xi_{m,n}(\omega) := \sup \big\{ y\in\bbR \,\big|\, u_m(np,x)\ge v(x-y) \text{ for all $x\in \bbR$} \big\}.
\]
This time the subadditive ergodic theorem  yields finite positive deterministic limits
\[
\xi = \lim_{n\to\infty} \frac{\xi_{0,n}(\omega)} n = \lim_{n\to \infty} \frac{\xi_{-n,0}(\omega)} n 
\]
for almost all $\omega\in\Omega$, and it again follows that $c:=p\xi$ is the asymptotic speed of $w_{\omega}$ as $|t|\to\infty$ for almost all $\omega\in\Omega$.



\end{document}